\documentclass[oneside,english]{amsart}
\usepackage[T1]{fontenc}
\usepackage[latin9]{inputenc}
\usepackage{textcomp}
\usepackage{amstext}
\usepackage{amsthm}
\usepackage{amssymb}

\makeatletter
\numberwithin{equation}{section}
\numberwithin{figure}{section}
\theoremstyle{plain}
\newtheorem{thm}{\protect\theoremname}
\theoremstyle{plain}
\newtheorem{cor}[thm]{\protect\corollaryname}
\theoremstyle{plain}
\newtheorem{prop}[thm]{\protect\propositionname}
\theoremstyle{definition}
\newtheorem{example}[thm]{\protect\examplename}
\theoremstyle{plain}
\newtheorem{lem}[thm]{\protect\lemmaname}

\def\makebbb#1{
    \expandafter\gdef\csname#1\endcsname{
        \ensuremath{\Bbb{#1}}}
}\makebbb{R}\makebbb{N}\makebbb{Z}\makebbb{C}\makebbb{H}\makebbb{E}\makebbb{H}\makebbb{P}\makebbb{B}\makebbb{Q}\makebbb{E}

\usepackage{babel}

\usepackage{babel}

\makeatother

\usepackage{babel}

\makeatother

\usepackage{babel}
\providecommand{\corollaryname}{Corollary}
\providecommand{\examplename}{Example}
\providecommand{\lemmaname}{Lemma}
\providecommand{\propositionname}{Proposition}
\providecommand{\theoremname}{Theorem}

\begin{document}
\title{Convexity of the K-energy and Uniqueness of Extremal Metrics -- An
Expository Introduction}
\author{Robert J. Berman, Bo Berndtsson}
\begin{abstract}
This article is an expository introduction to our paper \emph{Convexity
of the K-energy and Uniqueness of Extremal Metrics.} We present the
main ideas behind the proof that Mabuchi\textquoteright s K-energy
functional is convex along weak geodesics in the space of K\"ahler potentials
and explain how this leads to the uniqueness of constant scalar curvature
K\"ahler metrics and extremal metrics up to automorphisms. The emphasis
is on the conceptual framework and key techniques.
\end{abstract}

\maketitle

\section{Introduction}

Let $X$ be a compact orientable real surface and fix a complex structure
$J$ on $X$ (i.e. an endomorphism of the tangent bundle of $TX$
whose square is minus the identity). By the classical uniformization
theorem, $X$ can be endowed with a Riemannian metric $g$ that is
$J-$invariant and has constant Gauss curvature. Normalizing $g$
so that the associated volume form gives unit total volume to $X,$
such a metric is uniquely determined, modulo the action of the group
$\mbox{Aut}_{0}(X)$ of biholomorphisms of $X$ homotopic to the identity.
In this way, one thus obtains a \emph{canonical }metric $g$ on $(X,J).$
A far-reaching generalization of such canonical metrics to compact
complex manifolds $(X,J)$ of arbitrary complex dimension $n$ was
proposed by Calabi in his seminal work \cite{ca-2}: \emph{K\"ahler
metrics with constant scalar curvature. }Recall that a Riemannian
metric $g$ on a complex manifold $(X,J)$ is said to be\emph{ K\"ahler}
if $g$ is $J$-invariant and the associated two form $\omega:=g(,J\cdot)$
is closed. Following standard practice we will identify $g$ with
the corresponding K\"ahler form $\omega.$ Calabi raised the problem
of existence and uniqueness of K\"ahler metrics $\omega$ with constant
scalar curvature in a given cohomology class in $H^{2}(X,\R)$ (in
the classical case when $n=1,$ $\omega$ is automatically closed
and fixing the cohomology class simply amounts to fixing the volume).
One of the main results of our work \cite{b-b1} is the proof of uniqueness,
in general:
\begin{thm}
\label{thm:uniqueness intro} Given any two cohomologous K\"ahler metrics
$\omega_{0}$ and $\omega_{1}$ on $X$ with constant scalar curvature
there exists an an element F in the group $\mbox{Aut}_{0}(X)$ of
biholomorphisms of $X$ homotopic to the identity, such that $\omega_{0}=F^{*}\omega_{1}.$ 
\end{thm}

More generally, the previous uniqueness result is established for
Calabi's \emph{extremal metrics}, i.e. the critical points of the
functional defined by the $L^{2}-$norm of the scalar curvature $R_{\omega}$
(see Section \ref{subsec:Generalization-to-extremal}). But here we
will, for simplicity, mainly focus on K\"ahler metrics with constant
scalar curvature. The key ingredient in the proof of Theorem \ref{thm:uniqueness intro}
is a convexity result for Mabuchi K-energy functional on the space
of all K\"ahler metrics in the cohomology class $[\omega_{0}]$ of a
fixed K\"ahler metric $\omega_{0}.$ In order to state this result,
first recall that any K\"ahler form $\omega$ in $[\omega_{0}]$ may
be expressed as 
\[
\omega=\omega_{u}:=\omega+dd^{c}u,\,\,\,\,\,\,(d^{c}:=\frac{1}{2}J^{*}d\implies dd^{c}=i\partial\bar{\partial})
\]
 for some $u\in C^{\infty}(X)$ - called the \emph{K\"ahler potential}
of $\omega$ - uniquely determined up to an additive constant. The
K\"ahler condition on $\omega$ translates to the positive property
$\omega_{u}>0$ on $u.$ Setting 
\[
\mathcal{H}(X,\omega_{0}):=\left\{ u\in C^{\infty}(X):\,\omega_{u}>0\right\} 
\]
 we can thus identify the space of all K\"ahler metrics in $[\omega_{0}]$
with $\mathcal{H}(X,\omega_{0})/\R.$ As shown by Mabuchi \cite{mab-1},
the infinite dimensional space $\mathcal{H}(X,\omega_{0})$ may be
endowed with a natural Riemannian metric, where the squared norm of
a tangent vector $v\in C^{\infty}(X)$ at $u$ is defined by 
\begin{equation}
\left\Vert v\right\Vert _{|u}^{2}:=\int_{X}v^{2}\omega_{u}^{n}.\label{eq:Mab Rieman metric}
\end{equation}
The Mabuchi K-energy functional $\mathcal{M}$ on the Riemannian manifold
$\mathcal{H}(X,\omega_{0})$ \cite{mab0} is uniquely defined, modulo
an additive constant, by the property that is gradient is the negative
of the normalized scalar curvature of the corresponding K\"ahler metric:
\begin{equation}
\nabla\mathcal{M}_{|u}:=-(R_{\omega_{u}}-\bar{R}),\label{eq:def of mab as gradient intro}
\end{equation}
 where $\bar{R}$ denotes the average scalar curvatures which, for
cohomology reasons, is a topological invariant. $\mathcal{M}$ descends
to $\mathcal{H}(X,\omega_{0})/\R$ (since $\left\langle \nabla\mathcal{M},1\right\rangle _{u}=0)$
and thus induces a function on the space of all K\"ahler metrics in
$[\omega_{0}].$ 

As shown by Mabuchi \cite{mab-1}, the functional $\mathcal{M}$ on
$\mathcal{H}(X,\omega_{0})$ is convex along geodesics $u_{t}$ in
the Riemannian manifold $\mathcal{H}(X,\omega).$ However, given $u_{0}$
and $u_{1}$ in $\mathcal{H}(X,\omega_{0})$ there may be no geodesic
$u_{t}$ connecting them (see \cite{l-v} for recent counterexamples).
Nevertheless, as demonstrated by Chen \cite{c0}, there always exists
a unique \emph{weak }geodesic $u_{t}$ connecting $u_{0}$ and $u_{1},$
which has the property that $u_{t}$ is in 
\[
\mathcal{H}_{1,1}(X,\omega_{0})=\left\{ u\in L^{1}(X):\,\omega_{u}\geq0,\,\,\,\omega_{u}\in L_{\text{loc}}^{\infty}\right\} .
\]
(see Section \ref{subsec:Weak-geodesics}). Chen further showed \cite{c1}
that the Mabuchi functional $\mathcal{M}$ admits an explicit formula
that remains well-defined on $\mathcal{H}_{1,1}(X,\omega_{0})$ (although
the original defining property of $\mathcal{M}(u)$ requires that
$\omega_{u}$ be positive and $C^{2}-$smooth). It was conjectured
by Chen that $\mathcal{M}(\phi_{t})$ is convex along any weak geodesic
as above \cite{c1} (the case when $c_{1}(X)\leq0$ was settled by
in \cite{c1}). The main result in \cite{b-b1} establishes Chen's
conjecture in full generality: 
\begin{thm}
\label{thm:main intro}For any given K\"ahler metric $\omega,$ the
Mabuchi functional $\mathcal{M}$ on $\mathcal{H}(X,\omega)$ is convex
along the weak geodesic $u_{t}$ connecting any two given elements
$u_{0}$ and $u_{1}$ in $\mathcal{H}(X,\omega).$
\end{thm}

Apart from the uniqueness result in Theorem \ref{thm:uniqueness intro},
another application of the previous theorem is the following result,
which follows directly from ``sub-slope property'' of convex functions:
\begin{cor}
\label{cor:csck min mab etc intro}Any K\"ahler metric with constant
scalar curvature metric minimizes the corresponding Mabuchi functional.
More precisely, the following inequality holds 
\begin{equation}
\mathcal{M}(u_{1})-\mathcal{M}(u_{0})\geq-d(u_{1},u_{0})\sqrt{\mathcal{C}(u_{0})},\label{eq:chens ineq in intro}
\end{equation}
 for any two K\"ahler potentials $u_{0}$ and $u_{1}$ on a K\"ahler manifold
$(X,\omega),$ where $d$ is the distance function corresponding to
the Mabuchi metric and $\mathcal{C}$ denotes the Calabi energy, i.e.
$\mathcal{C}(u):=\int(R_{\omega_{u}}-\bar{R})^{2}\omega_{u}^{n}.$
\end{cor}

Since the publication of our work \cite{b-b1}, there has been a series
of further developments building on the convexity of the Mabuchi functional
along weak geodesics. For example, in \cite{d-r,b-d-l2} it is shown
that the existence of a K\"ahler metric in $[\omega_{0}]$ with constant
scalar curvature implies that the Mabuchi functional $\mathcal{M}$
is coercive (proper) modulo the group $\mbox{Aut}_{0}(X);$ the converse
was subsequently established in \cite{c-cII}. Moreover, building on the latter result, variants of the Yau-Tian-Donaldson conjecture were very recently established in \cite{bj}, \cite{dz}, linking the existence of a K\"ahler metric with constant scalar curvature in the first Chern class $c_1(L)$ of an ample line bundle $L$ to algebro-geometric conditions on $(X, L)$.

\subsection{Relation to previous results}

In the case when the first Chern class $c_{1}(X)$ of $X$ is trivial
the uniqueness of constant scalar curvature metrics $\omega$ in $[\omega_{0}]$
(which, in this case, have vanishing Ricci curvature) was established
by Calabi \cite{ca}, using integration by parts. The case when $[\omega]=-c_{1}(X)$
(i.e. when $\omega$ has constant negative Ricci curvature) is due
to Yau and Aubin \cite{y,au}, using the maximum principle. In fact,
in these cases absolute uniqueness holds, i.e. the automorphism $F$
is the identity, as well as existence \cite{y,au}. When $[\omega]/2\pi=c_{1}(X)$
(i.e. when $\omega$ has constant positive Ricci curvature) the uniqueness
result in Theorem \ref{thm:uniqueness intro} is due to Bando-Mabuchi
\cite{b-m}, who used Aubin's method of continuity. The case when
$[\omega]/2\pi$ is integral, i.e. $[\omega]/2\pi=c_{1}(L)$ for some
ample line bundle $L\rightarrow X$ was shown by Donaldson \cite{d00},
when $\mbox{Aut}_{0}(X,L)$ is trivial. The proof uses approximation
with, so called, balanced metrics attached to high tensor powers of
the line bundle $L$ that are unique, when $\mbox{Aut}_{0}(X,L)$
is trivial (see also \cite{mab2}, concerning the case when $\mbox{Aut}_{0}(X,L)$
is non-trivial). The general uniqueness result in Theorem \ref{thm:uniqueness intro}
is shown by Chen-Tian in \cite{c-t}, relying on the partial regularity
theory for weak geodesics developed in \cite{c-t}. However, by the
counterexamples of Ross and Witt Nystr\"om \cite{r-w}, the partial
regularity results do not hold as stated in \cite{c-t}, so it seems
that the proof in \cite{c-t} is not complete. 

The key ingredient in the proof of the convexity stated in Theorem
\ref{thm:main intro} is a new local positivity property of the relative
canonical line bundle $K_{X\times D/D}$ along the one-dimensional
current $S:=(dd^{c}\Phi)^{n}$ on $X\times D,$ where $X$ and $D$
denote domains in $\C^{n}$ and $\C$ respectively and $\Phi$ is
a function on $X\times D$ solving the complex Monge-Amp\`ere equation
\[
(dd^{c}\Phi)^{n+1}=0\,\,\text{on \ensuremath{X\times D}}
\]
such that $dd^{c}\Phi\in L_{\text{loc}}^{\infty}$ (see \cite[Section 3.2]{b-b1}
and Theorem \ref{thm:main text} below). This result may be viewed
as a generalization of a positivity property of Monge-Amp\`ere foliations
due to Bedford-Burns \cite{b-bu} and later extended by Chen-Tian
\cite{c-t}, since $S$ can be interpreted as an average of the leaves
of such a foliation - when such a foliation exists. However, a central
feature of our approach is that it does not rely on the existence
of any Monge-Amp\`ere foliation, which in fact need not exist in general
\cite{r-w,r-w2}. Instead, our proof is based on the plurisubharmonic
variation of local Bergman kernels \cite{bern0}. In the subsequent
work \cite{c-l-p} of Chen-Paun-Li, the local Bergman kernels used
in \cite{b-b1} were replaced by solutions to global Monge--Amp\`ere
equations (which can be interpreted as global \textquotedblleft transcendental\textquotedblright{}
Bergman kernels \cite{b1}).

\section{The proof of convexity (Theorem \ref{thm:main intro})}

We continue with the notations in the introduction. In particular,
$X$ denotes a compact complex manifold of dimension $n,$ endowed
with a K\"ahler form $\omega_{0}.$ 

\subsection{Preliminaries}

In order to explain the proof of Theorem \ref{thm:main intro} we
start by recalling the notion of weak geodesics and the explicit formula
for the Mabuchi functional $\mathcal{M}.$

\subsubsection{\label{subsec:Weak-geodesics}Weak geodesics}

A curve $u_{t}$ in $\mathcal{H}(X,\omega_{0})$ is a geodesic iff
it satisfies the following equation \cite{mab-1}: 
\begin{equation}
\ddot{u_{t}}=\left|\overline{\partial}(\dot{u_{t}})\right|_{\omega_{u_{t}}}^{2},\label{eq:geod eq real case}
\end{equation}
As observed independently by Semmes \cite{se} and Donaldson \cite{do00},
the geodesic equation for $u_{t}$ on $X\times[0,1]$ can be reformulated
- after complexifying the parameter $t$ - as a complex Monge-Amp�re
equation in $X\times D,$ where $D$ is the strip in $\C$ where $\Re\tau\in]0,1[.$
More precisely, denoting by $\pi$ the natural projection from $X\times D$
to $X$ and setting $U(x,\tau):=u_{\Re\tau}(x),$ the geodesic equation
becomes 
\begin{equation}
(\pi^{*}\omega+dd^{c}U)^{n\text{+1}}=0,\label{eq:ma eq for geod intro}
\end{equation}
Chen \cite{c0} (see also \cite{bl} for some refinements) showed
that for any smoothly bounded domain $\Omega$ in $\C$ the corresponding
Dirichlet problem on $X\times\Omega$ admits a unique solution $U$
such that $\pi^{*}\omega+dd^{c}U$ is a positive current with coefficients
in $L^{\infty},$ satisfying the equation \ref{eq:ma eq for geod intro}
almost everywhere. In particular, when $\Omega$ is an annulus in
$\C$ (which is covered by $D)$ this construction yields a \emph{weak
geodesic} $u_{t}$ in the extended space $\mathcal{H}_{1,1}(X,\omega_{0})$
of all functions $u$ on $X$ such that $\omega_{u}$ is a positive
current with coefficients in $L^{\infty}.$ 

\subsubsection{The explicit formula for $\mathcal{M}(u)$}

Set
\begin{equation}
\mathcal{E}(u):=\int_{X}\sum_{j=0}^{n}u\omega_{u}^{n-j}\wedge\omega_{0}^{j}\label{eq:def of energy of phi}
\end{equation}
Similarly, given a closed $(1,1)-$form (or current) $\alpha$ we
set

\begin{equation}
\mathcal{E}^{\alpha}(u):=\int_{X}u\sum_{j=0}^{n-1}\omega_{u}^{n-j}\wedge\omega_{0}^{j}\wedge\alpha\label{eq:def of T-energy of phi}
\end{equation}
We also recall that the \emph{entropy} of a measure $\mu$ relative
to a reference measure $\mu_{0}$ is defined as follows (in the case
when the measures in question have bounded densities):
\begin{equation}
H_{\mu_{0}}(\mu):=\int_{X}\log\left(\frac{\mu}{\mu_{0}}\right)\mu\label{eq:def of entr in smooth case}
\end{equation}

\begin{prop}
\label{prop:form for mab}Given a K\"ahler metric $\omega_{0}$ on $X$
with volume form $\mu_{0}:=\omega_{0}^{n}$ the following formula
holds for the Mabuchi functional on $\mathcal{H}(X,\omega_{0}):$
\begin{equation}
\mathcal{M}(u)=\left(\frac{\bar{R}}{n+1}\mathcal{E}(u)-\mathcal{E}^{\mbox{Ric}\ensuremath{\omega_{0}}}(u)\right)+H_{\mu_{0}}(\omega_{u}^{n}),\,\,\,\,\,\bar{R}:=\frac{nc_{1}(X)\cdot[\omega_{0}]^{n-1}}{[\omega_{0}]^{n}}\label{eq:formula for mab}
\end{equation}
\end{prop}

Following Chen \cite{c1}, we may use the previous proposition to
extend the functional $\mathcal{M}$ from $\mathcal{H}(X,\omega_{0})$
to the space $\mathcal{H}_{1,1}(X,\omega_{0})$ of all functions $u$
on $X$ such that $\omega+dd^{c}u\in L_{\text{loc}}^{\infty}$ and
$\omega+dd^{c}u\geq0.$ 

\subsection{A subharmonicity result for $\mathcal{M}$}

Let now $D$ be any domain in $\text{\C},$ $U(x,\tau)$ a function
on $X\times D$ and consider the family $u_{\tau}$ of functions on
$X,$ defined by $u_{\tau}:=U(\cdot,\tau).$ We will deduce the convexity
in Theorem \ref{thm:main intro} from the following result, where
a function $v(\tau)$ of one complex variable $\tau$ is called \emph{weakly
subharmonic} if $dd^{c}v\geq0$ in the sense of currents. 
\begin{thm}
\label{thm:main text} Let $D$ be an open domain in $\C$ and $U$
a function on $X\times D$ such that $dd^{c}U\in L_{\text{loc}}^{\infty},$
$\pi^{*}\omega_{0}+dd^{c}U\geq0$ and 
\begin{equation}
\left(\pi^{*}\omega_{0}+dd^{c}U\right)^{n+1}=0.\label{eq:MA eq in thm text}
\end{equation}
 Then, the Mabuchi functional $\mathcal{M}(u_{\tau})$ is weakly subharmonic
with respect to $\tau\in D.$
\end{thm}

The starting point of the proof of the previous theorem is the following,
essentially well-known, formula for the second order variation of
the Mabuchi functional:
\begin{equation}
d_{t}d_{t}^{c}\mathcal{M}(u_{t})=\int_{X}T,\,\,\,T:=dd^{c}\Psi\wedge(\pi^{*}\omega+dd^{c}U)^{n},\label{eq:form for dd of mab intro}
\end{equation}
 where $\int_{X}$ denotes the fiber-wise integral, i.e. the natural
map pushing forward a form on $X\times D$ to a form on $D$ and $\Psi$
denotes the local weight of the metric on the relative canonical line
bundle $K_{X\times D/D}\rightarrow X\times D$ induced by the metrics
$\omega_{u_{t}}$ on $TX,$ 
\[
\Psi_{t}:=\log(\frac{\omega_{u_{t}}^{n}}{i^{n^{2}}dz\wedge d\bar{z}}),\,\,\,\,dz:=dz_{1}\wedge\cdots\wedge dz_{n}
\]
expressed in local holomorphic coordinates on $X.$ The strategy of
the proof is to show that the integrand $T$ in formula \ref{eq:form for dd of mab intro}
is a non-negative top form on $M$ and in particular its push-forward
to $D$ is also non-negative, as desired. First observe that we can
locally write 
\[
\pi^{*}\omega+dd^{c}U=dd^{c}\Phi
\]
 for a local function $\Phi(t,z)=\phi_{t}(z),$ defined on the unit-ball
in $\C^{n},$ which is plurisubharmonic, i.e. $dd^{c}\Phi\geq0.$
In particular, locally, 
\[
\omega_{u_{t}}^{n}=(dd^{c}\phi_{t})^{n},\,\,\,\phi_{t}:=\Phi(\cdot,t).
\]
 When $\phi_{t}\in C_{\text{loc}}^{2}$ it follows from well-known
convergence results for Bergman kernels going back to H�rmander, Bouche
\cite{bo} and Tian \cite{t0} that the following point-wise limit
holds:
\[
\frac{(\frac{1}{2\pi}dd^{c}\phi_{t})^{n}}{i^{n^{2}}dz\wedge d\bar{z}}=\lim_{k\rightarrow\infty}k^{-n}B_{k\phi_{t}}(z),
\]
 where 
\[
B_{k\phi}:=K_{k\phi}e^{-k\phi},\,\,\,\,K_{k\phi}(z)=\sup_{f}\frac{\left|f(z)\right|^{2}}{\int_{|z|<1}\left|f\right|^{2}e^{-k\phi}i^{n^{2}}dz\wedge d\bar{z}}
\]
and the sup ranges over all holomorphic functions $f$ on the unit-ball
in $\C^{n}.$ As a consequence, if we make the further simplifying
assumption that $dd^{c}\phi_{t}>0,$ then the form $T$ can be locally
realized as the weak limit, as $k\rightarrow\infty,$ of the forms
$T_{k}$ defined by 
\[
T_{k}:=dd^{c}\log B_{k\phi_{t}}\wedge(dd^{c}\Phi)^{n},
\]
Now, by the positivity results in \cite{bern0} the function $\log K_{k\phi_{t}}$
is plurisubharmonic on $X\times D$ and hence 
\begin{equation}
dd^{c}\log B_{k\phi_{t}}=dd^{c}\log K_{k\phi_{t}}-kdd^{c}\Phi\geq0-kdd^{c}\Phi\label{eq:decomp of bergman intro}
\end{equation}
 Since the latter form vanishes when wedged with $(dd^{c}\Phi)^{n}$
(by the equation \ref{eq:MA eq in thm text}) this show that $T_{k}\geq0.$
Hence letting $k\rightarrow\infty$ reveals that $T\geq0$ which concludes
the proof of Theorem \ref{thm:main intro} under the simplifying assumption
that $\omega_{u_{t}}$ be continuous and strictly positive. The proof
of Theorem \ref{thm:main text} in the general case involves a truncation
procedure (to compensate the lack of strict positivity of the measures
$\omega_{u_{t}}^{n})$ and a generalization of the Bergman kernel
asymptotics used above to the case when the curvature form $dd^{c}\phi$
is merely in $L_{loc}^{\infty}.$ 

\subsection{Conclusion of the proof of Theorem \ref{thm:main intro}}

Applying Theorem \ref{thm:main text} to the case when $u_{t}$ is
a weak geodesic and using that $u_{t}$ is independent of the imaginary
part of $t$ shows that $d^{2}\mathcal{M}(u_{t})/dt^{2}\geq0$ on
$]0,1[,$ which means that $\mathcal{M}(u_{t})$ is weakly convex
on $]0,1[.$ Finally, the proof of Theorem \ref{thm:main intro} is
concluded by showing that $\mathcal{M}(u_{t})$ is continuous on $[0,1]$
and thus convex on $[0,1],$ by exploiting that $B_{k\phi_{t}}(z)$
is continuous wrt $t\in[0,1].$

\section{The proof of uniqueness (Theorem \ref{thm:uniqueness intro})}

The proof of Theorem \ref{thm:uniqueness intro} in \cite{b-b1} yields,
in fact, a slightly stronger uniqueness result where the group $\mbox{Aut}_{0}(X)$
is replaced by the subgroup $G$ whose Lie algebra $\mathfrak{g}$
consists of all holomorphic vector fields $V$ of type $(1,0)$ on
$X$ that admit a \emph{complex Hamiltonian} (wrt $\omega_{0})$,
i.e. a complex-valued smooth function $v$ on $X$ such that 
\begin{equation}
2i\overline{\partial}v=\omega_{0}(V,\cdot)\,\,\,\,\,\,\left(V=2\nabla^{1,0}v\right),\label{eq:V is complex gradient of v}
\end{equation}
 where $\nabla$ denotes the gradient wrt the metric $\omega_{0}.$
This condition is independent of the choice of $\omega_{0}$ (but
$v$ depends on $\omega_{0}).$ Moreover, $[\mathfrak{g},\mathfrak{g}]\subset\mathfrak{g}$
and thus $\mathfrak{g}$ is a Lie algebra. Note that when $v$ is
real-valued it is a \emph{Hamiltonian} for the imaginary part of $V,$
i.e. 
\begin{equation}
dv=\omega_{0}(\Im V,\cdot).\label{eq:v is Hamiltonian}
\end{equation}
The group $G$ is sometimes called the \emph{reduced automorphism
group} of $X.$ In this section we will outline the proof of the corresponding
uniqueness result:
\begin{thm}
\label{thm:uniq text}Given any two cohomologous K\"ahler metrics $\omega_{0}$
and $\omega_{1}$ on $X$ with constant scalar curvature there exists
an an element $g$ in $G$ such that the $\omega_{0}=g^{*}\omega_{1}.$ 
\end{thm}

It follows from the previous theorem that if $\omega_{0}$ has constant
scalar curvature, then $\text{Aut}_{0}(X)\omega_{0}=G\omega_{0},$
as previously shown by Calabi \cite{caII}.
\begin{example}
In general, $\mathfrak{g}$ consists of all holomorphic vector fields
that vanish at some point of $X$ \cite[Thm 1]{l-s}. For example,
when $X$ is a complex curve of genus zero, $\text{Aut}_{0}(X)$ is
non-trivial, but $\mathfrak{g}=\{0\}.$ Indeed, by the uniformization
theorem, $X\cong\C/\Lambda$ for some lattice $\Lambda$ of $\C$
and thus the non-trivial holomorphic vector fields on $X$ are zero-free
as they are induced by translations on $\C.$ The standard translationally
invariant K\"ahler form on $\C$ induces a K\"ahler form $\omega_{0}$
on $X$ with vanishing scalar curvature. The previous theorem thus
recovers the classical fact that $\omega_{0}$ is the unique K\"ahler
metric on $X$ with constant scalar curvature, with prescribed total
volume. This is consistent with the observation that any element in
$\mbox{Aut}_{0}(X)$ preserves $\omega_{0},$ since the elements in
$\mbox{Aut}_{0}(X)$ are covered by translations of $\C.$ 
\end{example}

By the convexity of $\mathcal{M}$ along weak geodesics (Theorem \ref{thm:main intro}),
the K\"ahler potential $u$ of any K\"ahler metric $\omega$ with constant
scalar curvature minimizes $\mathcal{M}$ on $\mathcal{H}(X,\omega_{0})$
(see Cor \ref{cor:csck min mab etc intro}). As a consequence, if
$u_{t}$ denotes the weak geodesic connecting two such K\"ahler potentials
$u_{0}$ and $u_{1},$ then $\mathcal{M}(u_{t})$ is affine on the
interval $[0,1].$ When $u_{t}$ is a genuine geodesic in $\mathcal{H}(X,\omega_{0}),$
it is well-known that this  implies that $\omega_{u_{t}}=g_{t}^{*}\omega_{0}$
for a one-parameter subgroup $g_{t}$ of $G.$ \cite{mab-1,do00}.
However, as shown in \cite{ber}, this implication fails when $u_{t}$
is merely a weak geodesic. The proof of Theorem \ref{thm:uniq text}
instead relies on a perturbation argument, involving the one-parameter
family of functionals 
\[
\mathcal{M}_{s}:=\mathcal{M}+s\mathcal{F},\,\,\,s\in\R
\]
 for an auxiliary functional $\mathcal{F}$ chosen to be strictly
convex along weak geodesics in $\mathcal{H}_{1,1}(X,\omega_{0})$
and invariant under the additive action of $\R:$ 
\[
\mathcal{F}(u):=\int_{X}u\mu-\mathcal{E}(u).
\]
 where $\mathcal{E}(u)$ is the functional defined by formula \ref{eq:def of energy of phi}
and $\mu$ is any fixed smooth volume form on $X$ whose volume coincides
with that of $\omega_{0}^{n}.$ In the case where $G$ is trivial,
i.e. $\mathfrak{g}=\{0\},$ one can show---via the implicit function
theorem---that for any critical point of $\mathcal{M}$, there exists
a nearby critical point of $\mathcal{M}$ (for small $s$). By strict
convexity, $\mathcal{M}_{s}$ admits at most one critical point, and
hence $\mathcal{M}$ does as well. This yields absolute uniqueness
of constant scalar curvature K\"ahler metrics in a fixed K\"ahler class.

When $G$ is non-trivial this direct perturbation argument breaks
down, since critical points of $\mathcal{M}_{s}$ need not exist near
those of $\mathcal{M}$; otherwise, we would obtain absolute uniqueness
rather than uniqueness modulo the action of $G.$ Nevertheless, one
can show that any critical point of $\mathcal{M}$ can be moved by
an element of $G$ to a new critical point that can be approximated
by critical points of $\mathcal{M}_{s}$ (for $s$ sufficiently small).
This yields uniqueness up to action of $G.$ The proof of this perturbation
result relies on a sophisticated version of the implicit function
theorem; see the first preprint version of \cite{b-b1} on ArXiv \cite{b-bv1}.
The published version \cite{b-b1} presents a simplified alternative
argument, based on \textquotedblleft almost critical points,\textquotedblright{}
which avoids using the implicit function theorem.

To explain the proof of Theorem \ref{thm:uniq text} in more detail,
let us first outline the proof given in the first preprint version
of \cite{b-b1} on ArXiv \cite{b-bv1}, that uses the implicit function
theorem in Banach spaces to prove the following general perturbation
result:
\begin{lem}
\label{lem:implc}Let $X$ be a compact manifold and $\mathcal{H}(X)$
an open subset $C^{\infty}(X)/\R,$ equipped with an action of a Lie
group $G,$ $\mathcal{M}$ a smooth invariant function on $\mathcal{H}(X)$$,$
$\mathcal{F}$ a smooth function on $\mathcal{H}(X)$ and $u_{0}$
a critical point of $\mathcal{M}.$ Assume that $\mathcal{M}$ is
$G-$invariant along the orbit of $G$ at $u_{0}$ and denote by $L_{0}$
the endomorphism of $C^{\infty}(X)/\R,$ defined as the derivative
at $u_{0}$ of the following map 
\[
\mathcal{H}(X)\rightarrow C^{\infty}(X),\,\,u\mapsto\nabla\mathcal{M}_{|u},
\]
 where the gradient is defined with respect to a fixed scalar product
$\left\langle \cdot,\cdot\right\rangle $ on $C^{\infty}(X)$ of the
form $\left\langle u,v\right\rangle :=\int uvdV$ for a fixed volume
form $dV$ on $X,$ so that $\nabla\mathcal{M}_{|u}$ can be identified
with an element of $C^{\infty}(X)/\R.$ Assume that 
\begin{itemize}
\item The kernel of $L_{0}$ coincides with with the infinitesimal orbit
of $G$ at $u_{0}$ (i.e. with the corresponding image of the Lie
algebra of $G)$.
\item The point $u_{0}$ is a non-degenerate local minimum for the restriction
$\mathcal{F}_{|Gu_{0}}$ of $\mathcal{F}$ to the orbit $Gu_{0}.$ 
\item $L_{0}$ is an elliptic operator of order $m$
\item $L_{0}$ extends to a (locally defined) smooth map between Sobolev
spaces $L^{p,r}(X)/\R$ and $L^{p,r-m}(X)/\R$ for $p$ and $r$ sufficiently
large.
\end{itemize}
Then there exists a curve $u_{t},$ defined for $t\in]-\epsilon,\epsilon[$
for some $\epsilon>0,$ of critical points to the perturbed functions
$\mathcal{M}+t\mathcal{F}$ such that $u_{t}\rightarrow u_{0}$ as
$t\rightarrow0.$
\end{lem}

In the present setup, $\mathcal{H}(X):=\mathcal{H}(X,\omega_{0})/\R,$
$\mathcal{M}$ is the Mabuchi functional and $G$ is the Lie group
defined in the beginning of the section with the right action on $\mathcal{H}(X,\omega_{0})/\R$
defined by
\[
g^{*}\omega=\omega_{0}+dd^{c}(u\cdot g).
\]
 Now fix the K\"ahler potential $u_{0}$ of a K\"ahler metric $\omega_{u_{0}}$
with constant scalar curvature. Since the K\"ahler potential of any
K\"ahler metric in $[\omega_{0}]$ with constant scalar curvature minimizes
$\mathcal{M}$ (by Cor \ref{cor:csck min mab etc intro}), $\mathcal{M}$
is $G-$invariant along the $G-$orbit $u_{0}\cdot G$ (in fact, $\mathcal{M}$
is $G-$invariant on all of $\mathcal{H}(X)$ since the Futaki invariants
vanish \cite{mab0}). Setting $dV:=\omega_{u_{0}}^{n},$ it is well-known
that the corresponding linear operator $L_{0}$ at $u_{0}$ is given
by the Lichnerowicz operator (up to multiplication by a harmless positive
number):
\begin{equation}
L_{0}=\mathcal{D}_{u_{0}}^{*}\mathcal{D}_{u_{0}},\,\,\,\,\,\,\,\mathcal{D}_{u}=\overline{\partial}\nabla^{1,0},\label{eq:L is Lich}
\end{equation}
 where $\mathcal{D}_{u_{0}}^{*}$ denotes the formal adjoint of the
operator $\mathcal{D}_{u_{0}}$ on $C^{\infty}(X,\C)$ (endowed with
the Hermitian product induced by $\omega_{u_{0}})$ \cite[Lemma 4.4]{sz}.
As a consequence, $\mathcal{D}_{u_{0}}^{*}\mathcal{D}_{u_{0}}$ is
a \emph{real }operator, i.e. it preserves $C^{\infty}(X,\R),$ when
$\omega_{u_{0}}$ has constant scalar curvature \cite{li,sz}. In
general, $\mathcal{D}_{u}v=0$ for $v\in C^{\infty}(X,\C)$ iff $v$
is the complex Hamiltonian of the holomorphic vector field $V:=2\nabla^{1,0}v$
(formula \ref{eq:V is complex gradient of v}). Moreover, when $\omega_{u_{0}}$
has constant scalar curvature,
\begin{equation}
v\in\text{ker}\mathcal{D}_{u_{0}}\iff v\in\ker\mathcal{D}_{u_{0}}^{*}\mathcal{D}_{u_{0}}\iff\Re v,\Im v\in\text{ker}\mathcal{D}_{u_{0}}\label{eq:v in kernel iff real and im}
\end{equation}
since $\mathcal{D}_{u_{0}}^{*}\mathcal{D}_{u_{0}}$ is a real operator. 

The structure of the $G-$orbit of the K\"ahler potential of any K\"ahler
metric with constant scalar curvature is described by the following
\begin{lem}
Assume that $\omega_{u_{0}}$ is a K\"ahler metric with constant scalar
curvature. Then $G$ is a complex reductive group and 
\begin{equation}
u_{0}\cdot G=u_{0}\cdot\exp(J\mathfrak{k}),\label{eq:G on u in terms of jk in lemma}
\end{equation}
 where $\mathfrak{k}$ denotes the Lie algebra of the maximally subgroup
$K$ of $G$ fixing $\omega_{u_{0}}.$ Moreover, for any given $W_{\R}$
in $\mathfrak{k}$ the corresponding curve 
\begin{equation}
u_{t}:=u_{0}\cdot\exp(tJW_{\R}),\,\,\,t\in\R\label{eq:def of u t in lemma}
\end{equation}
 in $\mathcal{H}(X,\omega_{0})/\R$ can be lifted to a geodesic in
$\mathcal{H}(X,\omega_{0})$ such that $\frac{1}{2}du_{t}/dt_{|t=0}$
is a Hamiltonian for $W_{\R}.$
\end{lem}

\begin{proof}
This result is essentially contained in \cite{mab-1}, but since it
is used implicitly in \cite{b-b1,b-bv1}, we explain the proof here.
In general, when $[\omega_{0}]$ contains a K\"ahler metrics with constant
scalar curvature $\omega,$ the group $G$ is a complex reductive
Lie group: it is the complexification of the maximally compact subgroup
$K$ on $G$ defined as the isometry group of $\omega$ \cite{li,sz}.
More precisely, under the identification $V\longleftrightarrow V_{\R},$
where a holomorphic vector field $V$ of type $(1,0)$ is identified
with its real part $V_{\R}$ (and thus the operators $i$ and $J$
are intertwined) it follows form formula \ref{eq:v in kernel iff real and im}
that 
\begin{equation}
\mathfrak{g}\longleftrightarrow\mathfrak{k}+J\mathfrak{k},\label{eq:reductivity}
\end{equation}
 where $\mathfrak{k}$ denote the Lie algebra of $K;$ $\mathfrak{k}$
consists of all holomorphic vector fields $W_{\R}$ that are Hamiltonian
wrt $\omega$. In particular, $K\cdot u_{0}=u_{0}$ in $\mathcal{H}(X,\omega_{0})/\R,$.
Formula \ref{eq:G on u in terms of jk in lemma} thus follows form
the polar decomposition of complex reductive Lie groups, saying that
$G=K\exp(J\mathfrak{k}).$ Finally, the fact that $u_{t}$ in $\mathcal{H}(X,\omega_{0})/\R,$
defined by formula \ref{eq:def of u t in lemma}, can be lifted to
a geodesic in $\mathcal{H}(X,\omega_{0}),$ follows from \cite[Thm 3.5]{mab-1},
using that $\text{\ensuremath{W_{\R}}}$ is a Hamiltonian vector field
for $\omega_{0}$ (see also \cite[Thm 3.1]{bern1} for a complex generalization).
As a courtesy to the reader we provide a proof here. Let us first
show that, if $u_{t}\in\mathcal{H}(X,\omega_{0})$ is normalized so
that $\mathcal{E}(u_{t})=0,$ then
\begin{equation}
\frac{1}{2}\dot{u_{t}}=h_{t},\label{eq:time derivat of u is Hamilton}
\end{equation}
 where $h_{t}$ is the unique Hamiltonian for $V_{\R}$ wrt $\omega_{u_{t}}(=\exp(tJW_{\R})^{*}\omega_{u_{0}})$
with vanishing average: 
\begin{equation}
dh_{t}=\iota_{W_{\R}}\omega_{u_{t}}(,\cdot),\,\,\,\int_{X}h_{t}\omega_{u_{t}}^{n}=0,\,\,\label{eq:Ham eq for h t}
\end{equation}
To this end, first note that 
\[
\omega_{u_{t}}=\exp(tJW_{\R})^{*}\omega_{u_{0}}\implies dd^{c}\dot{u_{t}}=d(\iota_{JW_{\R}}\omega_{u_{t}}),
\]
 by applying Cartan's formula to the Lie derivative of $\omega_{u_{t}}$
along $JW_{\R}.$ Now, applying $J$ to the Hamiltonian equation \ref{eq:Ham eq for h t}
gives 
\[
Jdh_{t}=\iota_{JW_{\R}}\omega_{u_{t}}.
\]
Combining the previous two equations thus yields $2dd^{c}h_{t}=dd^{c}\dot{u_{t}}$
(since $2d^{c}:=Jd),$ which means that $h_{t}-\dot{u_{t}}/2\in\R.$
This implies formula \ref{eq:time derivat of u is Hamilton}. Indeed,
differentiating $\mathcal{E}(u_{t})=0$ wrt $t$ reveals that $\dot{u_{t}}$
has vanishing average wrt $\omega_{u_{t}}^{n}$ and so has $h_{t},$
by definition. Finally, formula \ref{eq:time derivat of u is Hamilton}
gives, since $h_{t}=\exp(tJW_{\R})^{*}h_{0},$
\[
\frac{1}{2}\ddot{u_{t}}=\dot{h_{t}}=i_{JW_{\R}}dh_{t}=\omega_{u_{t}}(W_{\R},JW_{\R})=2\left|\overline{\partial}h_{t}\right|_{\omega_{t}}^{2}=2\left|\overline{\partial}\frac{1}{2}\dot{u_{t}}\right|_{\omega_{t}}^{2}
\]
which means that $u_{t}$ satisfies the geodesic equation \ref{eq:geod eq real case}.
\end{proof}
By the previous lemma, the kernel of $L_{0}$ coincides with the infinitesimal
orbit of $G$ at $u_{0}.$ Moreover, the analytic assumptions on $L_{0}$
in the third and fourth point of Lemma \ref{lem:implc} are also satisfied
with $m=4$ (since $\mathcal{D}_{u_{0}}=\Delta^{2}$ up to lower order
terms). The uniqueness result in Theorem \ref{thm:uniq text} thus
follows from Lemma \ref{lem:implc}, by replacing $u_{0}$ with the
critical point of $\mathcal{F}_{|u_{0}G}$ furnished by the following
proposition and using that $\mathcal{M}+s\mathcal{F}$ is strictly
convex along weak geodesics, when $s>0,$ as follows form combining
the convexity of $\mathcal{M}$ in Theorem \ref{thm:main intro})
with the strict convexity of $\mathcal{F}$ in the following proposition:
\begin{prop}
The functional $\mathcal{F}$ is strictly convex along weak geodesics.
Thus, if $\omega_{u_{0}}$ is a K\"ahler metric with constant scalar
curvature, then $\mathcal{F}$ is proper on the orbit $u_{0}\cdot G,$
As a consequence, $\mathcal{F}_{|u_{0}\cdot G}$ has a unique critical
point and it is a non-degenerate global minimum of $\mathcal{F}_{|u_{0}\cdot G}.$
\end{prop}

\begin{proof}
A standard direct computation reveals that $\mathcal{F}$ is strictly
convex along geodesics \cite[Prop 4.1]{b-b1}. Thus, by the previous
lemma, $\mathcal{F}_{|Gu_{0}}$ may be identified with a strictly
convex function on $\mathfrak{k}\cong\R^{n}.$ Moreover, when $\mu=\omega_{u_{0}}^{n}$
the corresponding functional $\mathcal{F}$ has a critical point at
$u_{0}.$ It thus follows from basic convex analysis that, in this
case, $\mathcal{F}_{|Gu_{0}}$ is proper. Since changing $\mu$ only
has the effect of shifting $\mathcal{F}$ with a constant, this means
that $\mathcal{F}_{|Gu_{0}}$ is proper, in general. In particular,
it admits a unique global minimum, which is non-degenerate.
\end{proof}
The proof of uniqueness given in the published version \cite[Prop 4.1]{b-b1}
makes a direct use of the first three points in Lemma \ref{lem:implc},
established above and thus bypasses the use of the implicit function
theorem. 

\subsection{Generalization to extremal K\"ahler metrics\label{subsec:Generalization-to-extremal}}

A K\"ahler metric $\omega$ in $[\omega_{0}]$ is called an \emph{extremal
K\"ahler metric} \cite{ca-2} if $\omega$ is a critical point of the
functional defined by the $L^{2}-$norm of the scalar curvature $R_{\omega}.$
As shown by Calabi \cite{ca-2}, this equivalently means that the
gradient of $R_{\omega}$ is the real part of a holomorphic vector
field $V$ of type $(1,0)$ - such a vector field $V$ is called an
\emph{extremal vector field. }The following generalization of Theorem
\ref{thm:uniqueness intro} is shown in \cite{b-b1}: 
\begin{thm}
Given any two cohomologous extremal K\"ahler metrics $\omega_{0}$ and
$\omega_{1}$ on $X$ there exists an element F in the group $\mbox{Aut}_{0}(X)$
such that $\omega_{0}=F^{*}\omega_{1}.$ 
\end{thm}

In fact, the proof in \cite{b-b1} yields a slightly stronger uniqueness
result. To explain this, first recall that an extremal vector field
$V$ is uniquely determined up to the action of $\mbox{Aut}_{0}(X)$
\cite{f-m}. Fix the extremal vector field $V$ and denote by the
subgroup $G_{V}$ of $\mbox{Aut}_{0}(X)$ defined as the Lie group
whose Lie algebra $\mathfrak{g}_{V}$ consists of all holomorphic
vector fields of type $(1,0)$ on $X$ that commute with $V$ and
admit a complex Hamiltonian.
\begin{thm}
Given any two cohomologous extremal K\"ahler metrics $\omega_{0}$ and
$\omega_{1}$ on $X$ with the same extremal vector field $V,$ there
exists an element $g$ in the group $G_{V}$ such that $\omega_{0}=g^{*}\omega_{1}.$
\end{thm}

The proof proceeds by generalizing the proof of Theorem \ref{thm:uniq text}
in the following way. First of all we may take the reference K\"ahler
metric $\omega_{0}$ to be invariant under the flow of $\Im V.$ Replace
$C^{\infty}(X)$ with the subspace $C_{V}^{\infty}(X)$ consisting
of all functions that are invariant under the flow of $\Im V$ and
replace $\mathcal{H}(X,\omega_{0})$ with the subspace $\mathcal{H}(X,\omega_{0})\cap C_{V}^{\infty}(X),$
that we denote by $\mathcal{H}_{V}(X,\omega_{0}).$ The group $G_{V}$
acts on $\mathcal{H}_{V}(X,\omega_{0})$ from the right, as before
and the K\"ahler potential of any extremal K\"ahler metric $\omega$ with
extremal vector field $V$ is in $\mathcal{H}_{V}(X,\omega_{0})$
(since $R_{\omega}$ is a Hamiltonian for $\Im V).$ There is a generalization
of the Mabuchi functional to $\mathcal{H}_{V}(X,\omega_{0}),$ that
we denote by $\mathcal{M}_{V},$ whose critical points are the K\"ahler
potentials of extremal K\"ahler metrics with extremal vector field $V$
\cite{gu,s-s}. Moreover, $\mathcal{M}_{V}-\mathcal{M}$ is affine
along weak geodesics, connecting given points in $\mathcal{H}_{V}(X,\omega_{0}).$
In this setup, the linearization $L_{0}$ of $\nabla\mathcal{M}_{V}$
at the K\"ahler potential $u_{0}$ of an extremal K\"ahler metrics with
extremal vector field $V$ is still of the form \ref{eq:L is Lich}
and the group $G_{V}$ is still the complexification of the compact
subgroup defined as the stabilizer of $\omega_{u_{0}}$ (see \cite[page 65]{sz}).
Hence, the rest of the proof proceeds essentially as before.


\begin{thebibliography}{10}
\bibitem{au}Aubin, T: Equations du type Monge-Amp\`ere sur les variet\'es
K\"ahleriennes compactes. Bull. Sci. Math. (2) 102 (1978), no. 1, 63--95

\bibitem{f-m}Akito Futaki and Toshiki Mabuchi, Bilinear forms and
extremal K\"ahler vector fields associated with K\"ahler classes, Math.
Ann. 301 (1995), no. 2, 199--210, 

\bibitem{b-m}Bando , S.H; Mabuchi , T: Uniqueness of Einstein K\"ahler
metrics modulo connected group actions. Algebraic geometry, Sendai,
1985, 11--40, Adv. Stud. Pure M ath., 10, North-Holland, Amsterdam,
1987.

\bibitem{b-bu}E.Bedford; D.Burns, Holomorphic mapping of annuli in
$\C^n$ and the associated extremal function, Ann. Sc. Norm. Super. Pisa
Cl. Sci. (4) 6 (1979), no. 3, 381--414

\bibitem{b1}Berman, R. J: From Monge--Amp\`ere equations to envelopes
and geodesic rays in the zero temperature limit. Math. Z. 291 (2019),
365--394.

\bibitem{b-b1}Berman, R.J; Berndtsson, B. Convexity of the K-energy
on the space of K\"ahler metrics and uniqueness of extremal metrics.
J. Amer. Math. Soc. 30 (2017), no. 4, 1165--1196

\bibitem{b-bv1}Berman, R.J; Berndtsson, B. Convexity of the K-energy
on the space of K\"ahler metrics and uniqueness of extremal metrics.
Preprint at https://arxiv.org/abs/1405.0401v1

\bibitem{ber}Berman, R.J: On the strict convexity of the K-energy.
PAMQ 15(4), 983--999 (2019)

\bibitem{bbgz}Berman, R.J; Boucksom, S; Guedj,V; Zeriahi: A variational
approach to complex Monge-Amp\`ere equations. Publications math. de
l'IH\'ES (2012): 1-67 , November 14, 2012

\bibitem{b-d-l2} R.J Berman, T Darvas, CH Lu: Regularity of weak
minimizers of the K-energy and applications to properness and K-stabiliy.
des Annales Scientifiques de l\'Ecole Normale Sup\'erieure. 53, fasc.
2 (2020).

\bibitem{bbj} Berman, R.J; Boucksom, S; Jonsson, M: A variational
approach to the Yau-Tian-Donaldson conjecture. J. Amer. Math. Soc.
34(2021), 605--652. 

\bibitem{bern0}Berndtsson, B: Subharmonicity properties of the Bergman
kernel and some other functions associated to pseudoconvex domains,
Ann. Inst Fourier (Grenoble) 56 (2006), 1633-1662.

\bibitem{bern1}Berndtsson, B: Long geodesics in the space of K\"ahler
metrics. Analysis Math., 48 (2) (2022), 377--392


\bibitem{bl}Blocki, Z: On geodesics in the space of K\"ahler metrics,
Proceedings of the \textquotedbl Conference in Geometry\textquotedbl{}
dedicated to Shing-Tung Yau (Warsaw, April 2009), in \textquotedbl Advances
in Geometric Analysis\textquotedbl , ed. S. Janeczko, J. Li, D. Phong,
Advanced Lectures in Mathematics 21, pp. 3-20, International Press,
2012 


\bibitem{bo}Bouche, T: Convergence de la m\'etrique de Fubini-Study
d'un fibr\'e lin\'eaire positif. Ann. Inst. Fourier (Grenoble) 40 (1990),
no. 1, 117--130

\bibitem{bj} Boucksom, S; Jonsson, M: On the Yau-Tian-Donaldson conjecture for weighted cscK metrics.
Preprint at https://arxiv.org/abs/arXiv:2509.15016.

\bibitem{ca}Calabi, E: The space of K\"ahler metrics. Proceedings of
the International Congress of Mathematics 1954, Vol 2 page 206- 1954
- Amsterdam

\bibitem{ca-2}Calabi, E: Extremal K\"ahler metrics, in Seminar on Differential
Geometry, volume 102 of Ann. of Math. Stud., pages 259--290, Princeton
Univ. Press, Princeton, N.J., 1982.

\bibitem{caII}E. Calabi. Extremal K\"ahler Metrics II. Differential
Geometry and Complex Analysis, pages 96-114, Springer, 1985.

\bibitem{c-c}E. Calabi and X.X. Chen. The space of K\"ahler metrics.
II. J. Differential Geom. , 61(2):173--193, 2002

\bibitem{c0}X.X. Chen, The space of K\"ahler metrics , J. Diff.
Geom. 56 (2000), 189-234

\bibitem{c1}X.X. Chen, On the lower bound of the Mabuchi energy and
its application , Int. Math. Res. Not. 2000, no. 12, 607-623

\bibitem{c-l-p}X.-X.Chen, L. Li and M. Paun: Approximation of weak
geodesics and subharmonicity of Mabuchi energy. Annales de la Facult\'e des sciences de Toulouse : Math\'ematiques, Serie 6, Volume 25 (2016) no. 5, pp. 935-957

\bibitem{c-t}X.X Chen; G.Tian. Geometry of K\"ahler metrics and foliations
by holomorphic discs. Publ. Math. Inst. Hautes \'Etudes Sci. No. 107
(2008), 1-107. 

\bibitem{c-cII}X.X. Chen and J. Cheng. On the constant scalar curvature
K\"ahler metrics (II)---Existence results. J. Amer. Math. Soc.
34 (2021), 937--1009.

\bibitem{d-r}T. Darvas, Y. Rubinstein.Tian's properness
conjectures and Finsler geometry of the space of K\"ahler metrics. J.
Amer. Math. Soc.30(2017), 347--387

\bibitem{dz}T. Darvas, K. Zheng:  A YTD correspondence for constant scalar curvature metrics. 
Preprint at https://arxiv.org/abs/2509.15173.

\bibitem{do00}S. K. Donaldson. Symmetric spaces, K\"ahler geometry
and Hamiltonian dynamics. In Northern California Symplectic Geometry
Seminar , volume 196 of Amer. Math. Soc. Transl. Ser. 2 , pages 13--33.
Amer. Math. Soc., Providence, RI, 1999

\bibitem{d00}Donaldson, S.K: Scalar curvature and projective embeddings,
I, J. Diff. Geom. 59 (2001), 479--522

\bibitem{gu}Daniel Guan, On modified Mabuchi functional and Mabuchi
moduli space of K\"ahler metrics on toric bundles,
Math. Res. Lett. 6 (1999), no. 5-6, 547--555,

\bibitem{fi}Fine, J. Constant scalar curvature K\"ahler
metrics on fibered complex surfaces. J. Differential Geom. 68(3):397--432,
2004.

\bibitem{l-v}Lempert and L.Vivas. Geodesics in the space of K\"ahler
metrics. Duke Math. J. Volume 162 , Number 7 (2013), 1369-138

\bibitem{l-s}C. LeBrun; S.R. Simanca: Extremal K\"ahler Metrics and
Complex Deformation Theory. Geometric and functional analysis (1994)
Vol.: 4, Issue: 3, page 298-336

\bibitem{li1}Li, C: G-uniform stability and K\"ahler-Einstein metrics
on Fano varieties. Invent. Math. 227 (2022), no. 2, 661--744.

\bibitem{li2}Li, C: Geodesic rays and stability in the cscK problem.
des Annales Scientifiques de l\'Ecole Normale Sup\'erieure. 55, fasc.
6 (2022)

\bibitem{li}Lichnerowicz, A.: Sur les transformations analytiques
des varietes kiihleriennes. C.R. Acad. Sci. Paris, 244 (1957), 3011-3014

\bibitem{mab0}Mabuchi, T., K -energy maps integrating Futaki invariants.
Tohoku Math. J. (2) 38 (1986), no. 4, 575-593

\bibitem{mab-1}Mabuchi, T: Some symplectic geometry on compact K\"ahler
manifolds. I. Osaka J. Math. , 24(2):227--252, 1987.

\bibitem{mab2}Toshiki Mabuchi, Uniqueness of extremal K\"ahler
metrics for an integral K\"ahler class, Internat.
J. Math. 15 (2004), no. 6, 531--546

\bibitem{r-w}J.Ross; D. Witt-Nystr\"om: Harmonic discs of solutions
to the complex homogeneous Monge-Amp\`ere equation, Publ. Math. Inst.
Hautes \'Etudes Sci. 122 (2015), 315--335, DOI 10.1007/s10240-015-0074-0

\bibitem{r-w2}J.Ross; D. Witt-Nystr\"om: The Dirichlet problem for
the complex homogeneous Monge-Amp\`ere equation. Modern geometry: a
celebration of the work of Simon Donaldson, Proceedings of Symposia
in Pure Mathematics 99 (eds I. Smith, V. Mun\~oz and R. P. Thomas;
American Mathematical Society, Providence, RI, 2018) 289--330

\bibitem{s-s}Santiago R. Simanca, A K-energy characterization of
extremal K\"ahler metrics, Proc. Amer. Math. Soc.
128 (2000), no. 5, 1531--1535, 

\bibitem{se}Semmes, S: Complex Monge-Amp`ere and symplectic manifolds.
Amer. J. Math. , 114(3):495--550, 1992

\bibitem{sz}G.Sz\'ekelyhidi: An introduction to extremal K\"ahler metrics,
volume 152 of Graduate Studies in Mathematics. American Mathematical
Society, Providence, RI, 2014

\bibitem{t0}Tian, G: On a set of polarized K\"ahler metrics on algebraic
manifolds. J. Differential Geom. 32 (1990), no. 1, 99--130. 

\bibitem{y}Yau, S-T: On the Ricci curvature of a compact K\"ahler manifold
and the complex Monge-Amp\`ere equation. I. Comm. Pure Appl. Math. 31
(1978), no. 3, 339--411
\end{thebibliography}
\end{document}